\documentclass[a4paper,12pt,reqno]{amsart}
\usepackage{amsmath,amsthm,amssymb}
\usepackage{graphicx, color}
\usepackage[numbers,sort&compress]{natbib}
\nonstopmode \numberwithin{equation}{section}
\newtheorem{definition}{Definition}[section]
\newtheorem{theorem}{Theorem}[section]

 \newtheorem{corollary}{Corollary}[section]

\newtheorem{remark}{Remark}[section]

\allowdisplaybreaks

\allowdisplaybreaks

\begin{document}

\title[Extended $(p,q)$-Mittag-Leffler function]
 {Extended $(p,q)$-Mittag-Leffler function and its properties}

\author[A. Kilicman, G. Rahman, K.S. Nisar, S. Mubeen]{Adem Kilicman*, Gauhar Rahman, Kottakkaran Sooppy  Nisar,\\
  Shahid Mubeen}

\address{Adem Kilicman: Department of Mathematics, University Putra, Malaysia Serdang, Malaysia}
\email{akilic@upm.edu.my}

\address{Gauhar Rahman:    Department of Mathematics, International Islamic  University, Islamabad, Pakistan}
\email{gauhar55uom@gmail.com}

\address{Kottakkaran Sooppy  Nisar:    Department of Mathematics, College of Arts and Science-Wadi Al dawser, 11991,
Prince Sattam bin Abdulaziz University, Alkharj, Saudi Arabia}
\email{ksnisar1@gmail.com, n.sooppy@psau.edu.sa}

\address{Shahid Mubeen: Department of Mathematics, University of Sargodha,   Sargodha, Pakistan}
\email{smjhanda@gmail.com}


\thanks{*Corresponding author}

\begin{abstract}
 In this study our aim to define the extended $(p,q)$-Mittag-Leffler(ML) function by using extension of beta functions and to obtain the integral representation of new function. We also take the Mellin transform of this new function in terms of Wright hypergeometric function. Extended fractional derivative of the classical Mittag-Leffler(ML)  function leads the extended  $(p,q)$-Mittag-Leffler(ML) function.
\end{abstract}

\maketitle

\section{Introduction and Preliminaries }\label{Intro}

The Mittag-Leffler(ML) function occurs naturally in many real world proplems, especially in the solution of fractional integro-differential equations having the arbitrary order. The importance of such functions in physics and engineering is steadily increasing. Some applications of the Mittag-Leffler(ML) is carried out in the study of kinetic equation, study of Lorenz system, random walk, Levy flights and complex systems and also in applied problems such as fluid flow and electric network.

 We begin with the Prabhakar \cite{R16} Mittag-Leffler(ML) function, which is defined by
\begin{eqnarray}\label{a3}
E_{\rho,\sigma}^{\gamma}(z)=\sum\limits_{n=0}^{\infty}\frac{(\gamma)_n}{\Gamma(\rho n+\sigma)}\frac{z^n}{n!}, \ z, \sigma\in\mathbb{C};\ {\rm and} \ \mathfrak{R}(\rho)>0,
\end{eqnarray}
where $(\gamma)_n$ is the Pochhammer Symbol and given by:
\begin{eqnarray*}
(\gamma)_n=\gamma(\gamma+1)\cdots(\gamma+n-1),\end{eqnarray*} for $(n\in\mathbb{N}, \gamma\in\mathbb{C})$. In particular if $n=0$ then $(\gamma)_n =1$.

In theory of special functions, the applications and importance of Mittag-Leffler(ML) functions were studied by many researchers and their extensions are found in (\cite{R6}, \cite{R21}-\cite{R24}).
Shukla and Prajapati \cite{CP} (see also \cite{SZ}) have further introduced the function  $E_{\rho, \sigma}^{\delta,q}(z)$, defined as:
\begin{eqnarray}\label{a4}
E_{\rho, \sigma}^{\delta,q}(z)=\sum\limits_{n=0}^{\infty}\frac{(\delta)_{nq}}{\Gamma(\rho n+\sigma)}\frac{z^n}{n!},
\end{eqnarray}
where $z,\sigma,\delta\in\mathbb{C}$; and $\mathfrak{R}(\rho)>0$; $q>0$. By considering (\ref{a4}) for particular parameters one can deduce several different special functions. \\
  \"{O}zarslan and Yilmaz \cite{OY} have defined the following extended Mittag-Leffler(ML) function $E_{\rho,\sigma}^{\delta;c}(z;p)$ by
\begin{eqnarray}\label{a5}
E_{\rho,\sigma}^{\delta;c}(z;p)=\sum\limits_{n=0}^{\infty}\frac{\beta_p(\delta+n, c-\delta)}{\beta(\delta, c-\delta)}\frac{(c)_n}{\Gamma(\rho n+\sigma)}\frac{z^n}{n!},
\end{eqnarray}
where $p\geq0$, $\mathfrak{Re}(c)>\mathfrak{Re}(\delta)>0$ and $\beta_p(x,y)$ is extended beta function, see \cite{CQ,FM} given by
\begin{eqnarray}\label{dd}
\beta_p(x,y)=\int\limits_{0}^{1}t^{x-1}(1-t)^{y-1}e^{-\frac{p}{t(1-t)}}dt,
\end{eqnarray}

where $\mathfrak{Re}(p)>0$, $\mathfrak{Re}(x)>0$ and $\mathfrak{Re}(y)>0$.\\

Very recently Rahman et al. (\cite{Rahman1}, \cite{Rahman2} and \cite{Rahman3}) defined the fractional integrals and differentials formulas and pathway integral formulas of extended Mittag-Leffler(ML) functions.\\

 Choi et al. \cite{Choi2014} have defined the following extension of beta function by
\begin{eqnarray}\label{EEbeta}
\beta(x,y;p,q)= \int\limits_{0}^{\infty}t^{x-1}(1-t)^{y-1}e^{-\frac{p}{t}-\frac{q}{1-t}}dt
\end{eqnarray}
(where $\Re(p)>0, \Re(q)>0, \Re(x)>0$, and $\Re(y)>0$). \\
 Obviously, when $p=q$, then (\ref{EEbeta}) will lead to extended beta function in (\ref{dd}). In particular, if $p=q=0$, then  (\ref{EEbeta}) will lead to the classical beta function.

In the present work, we define further extension of Mittag-Leffler(ML) function as:
\begin{eqnarray}\label{Fextended}
E_{\alpha, \beta}^{\gamma, c}(z;p,q)=\sum\limits_{n=0}^{\infty}\frac{\beta(\gamma+n, c-\gamma;p,q)}{\beta(\gamma,c-\gamma)}\frac{(c)_n}{\Gamma(\alpha n+\beta)}\frac{z^n}{n!}
\end{eqnarray}
where $p, q\geq0$, $\Re(c)>\Re(\gamma)>0$ while $\beta(x,y;p,q)$ is considered as extension of extended beta function defined in
(\ref{EEbeta}).
\begin{remark} In particular 
(1) If setting $p=q$ in (\ref{Fextended}), then deduce to the extended Mittag-Leffler(ML) function as in the equation (\ref{a5}).\\
(ii) If setting  $p=q=0$ in (\ref{Fextended}), then reduce to Mittag-Leffler(ML) function as in (\ref{a3}).
\end{remark}
\section{Properties of extended $(p,q)$-Mittag-Leffler(ML) function}

Next, we define  the integral representations of \eqref{Fextended}.
\begin{theorem}\label{tha}
Let $c, \alpha, \beta, \gamma\in\mathbb{C}$, $\Re(c)>\Re(\gamma)>0$, $\Re(\alpha)>0$ and $\Re(\beta)>0$. Then the following integral representation hold true,
\begin{eqnarray}\label{integral}
E_{\alpha,\beta}^{\gamma,c}(z;p,q)=\frac{1}{\beta(\gamma,c-\gamma)}\int_0^1t^{\gamma-1}(1-t)^{c-\gamma-1}
\exp\Big(-\frac{p}{t}-\frac{q}{(1-t)}\Big)E_{\alpha,\beta}^{c}(tz)dt.
\end{eqnarray}
\end{theorem}
\begin{proof}
Using (\ref{EEbeta}) in equation (\ref{Fextended}), we have
\begin{eqnarray*}
E_{\alpha,\beta}^{\gamma,c}(z;p,q)=\sum\limits_{n=0}^{\infty}\Big\{\int_0^1t^{\gamma+n-1}(1-t)^{c-\gamma-1}
\exp\Big(-\frac{p}{t}-\frac{q}{(1-t)}\Big)dt\Big\}\\
\times\frac{(c)_n}{\beta(\gamma,c-\gamma)}\frac{z^n}{\Gamma(\alpha n+\beta)n!}.
\end{eqnarray*}
which can be written as
\begin{eqnarray*}
E_{\alpha,\beta}^{\gamma,c}(z;p,q)=\int_0^1t^{\gamma-1}(1-t)^{c-\gamma-1}
\exp\Big(-\frac{p}{t}-\frac{q}{(1-t)}\Big)\\
\times\sum\limits_{n=0}^{\infty}
\frac{(c)_n}{\beta(\gamma,c-\gamma)}\frac{(tz)^n}{\Gamma(\alpha n+\beta)n!}dt.
\end{eqnarray*}
Using (\ref{a3}) in above equation, we get the desired integral representation.
\end{proof}
\begin{corollary}
Substituting $t=\frac{u}{1+u}$ in Theorem \ref{tha}, we get
\begin{eqnarray}
E_{\alpha,\beta}^{\gamma,c}(z;p,q)=\frac{1}{\beta(\gamma,c-\gamma)}\int_0^\infty\frac{u^{\gamma-1}}{(u+1)^c}
\exp\Big(-\frac{p(1+u)}{u}-q(1+u)\Big)\notag\\
\times E_{\alpha,\beta}^{c}\Big(\frac{uz}{1+u}\Big)du.
\end{eqnarray}
\end{corollary}
\begin{corollary}
Setting $t=\sin^2\theta$ in Theorem \ref{tha}, we get following integral representation
\begin{eqnarray}
E_{\alpha,\beta}^{\gamma,c}(z;p,q)=\frac{1}{\beta(\gamma,c-\gamma)}\Big[2\int_0^\infty\sin^{2\gamma-1}
\theta\cos^{2c-1}\theta\exp
\Big(-\frac{p}{\sin^2\theta}-\frac{q}{\cos^2\theta}\Big)\Big]\notag\\
\times E_{\alpha,\beta}^{c}\Big(z\sin^2\theta\Big)d\theta.
\end{eqnarray}
\end{corollary}
  Kurulay and Bayram \cite{Kurulay} defined the following,
\begin{equation}\label{GF}
E_{\alpha,\beta}^{c}(tz)=\beta E_{\alpha,\beta+1}^{c}(tz)+\alpha z\frac{d}{dz} E_{\alpha,\beta+1}^{c}(tz).
\end{equation}
By using \eqref{GF} and \eqref{integral}, we obtained the following relation:
\begin{corollary}
Let $p, q\geq0$, $\Re(c)>\Re(\gamma)>0$, $\Re(\alpha)>0$, $\Re(\beta)>0$, then following relation holds:
\begin{eqnarray}
E_{\alpha,\beta}^{\gamma,c}(z;p,q)=\beta E_{\alpha,\beta+1}^{\gamma,c}(z;p,q)+\alpha z\frac{d}{dz}E_{\alpha,\beta+1}^{\gamma,c}(z;p,q)
\end{eqnarray}
\end{corollary}

In the next theorem, we apply the Mellin transforms on \eqref{Fextended} and obtained the result in form of Wright hypergeometric function which is defined in (see \cite{Wright1935}-\cite{Wrigt1935})  by
$$_p\Psi_{q}(z)=\quad_p\Psi_{q}
                                 \left[
                                   \begin{array}{ccc}
                                     (\alpha_i, \mu_i)_{1,p} &  &  \\
                                      &  & ;z \\
                                     (\beta_j, \lambda_j)_{1,q} &  &  \\
                                   \end{array}
                                 \right]$$
\begin{eqnarray}\label{6}
&=&\sum\limits_{n=0}^{\infty}\frac{\Gamma(\alpha_{1}+\mu_{1}n)\cdots\Gamma(\alpha_{p}+\mu_{p}n)}
{\Gamma(\beta_{1}+\lambda_{1}n)\cdots
\Gamma(\beta_{q}+\lambda_{q}n)}\frac{z^{n}}{n!}
\end{eqnarray}
where $\beta_{r}$ and $\alpha_{s}$  belongs to $\mathbb{R^{+}}$ such that
\begin{eqnarray*}
1+\sum\limits_{s=1}^{q}\lambda_{s}-\sum\limits_{r=1}^{p}\mu_{r}\geq0.
\end{eqnarray*}

\begin{theorem}\label{thb}
The Mellin transform of \eqref{Fextended} given by
\begin{eqnarray}\label{Mell}
\mathfrak{M}\Big\{E_{\alpha,\beta}^{\gamma,c}(z;p,q);p\rightarrow s,q\rightarrow r\Big\}
=\frac{\Gamma(s)\Gamma(r)\Gamma(c+r-\gamma)}{\Gamma(\gamma)\Gamma(c-\gamma)}\notag\\
\times _2\Psi_2\left[
         \begin{array}{cc}
           (c,1), (\gamma+s,1), & \\
            & ,z\\
           (\beta,\gamma),(c+s+r,1),&  \\
         \end{array}
       \right]
\end{eqnarray}
where $p,q\geq0$, $\Re(c)>\Re(\gamma)>0$, and $\Re(\alpha)>0$,  $\Re(\beta)>0$.
\end{theorem}
\begin{proof}
Applying the Mellin transform to \eqref{Fextended},
\begin{eqnarray}\label{Mellin}
\mathfrak{M}\Big\{E_{\alpha,\beta}^{\gamma,c}(z;p,q);p\rightarrow s,q\rightarrow r\Big\}=
\int_{0}^{\infty}\int_{0}^{\infty}p^{s-1}q^{r-1}E_{\alpha,\beta}^{\gamma,c;\lambda,\rho}(z;p,q)dpdq.
\end{eqnarray}
and using (\ref{integral}) in (\ref{Mellin}), we get
\begin{align}\label{Mellin1}
&\mathfrak{M}\Big\{E_{\alpha,\beta}^{\gamma,c}(z;p,q);p\rightarrow s,\rightarrow r\Big\}\notag\\
&=\frac{1}{\beta(\gamma,c-\gamma)}\int_0^\infty\int_{0}^{\infty} p^{s-1}q^{r-1}
\Big\{\int_0^1t^{\gamma-1}(1-t)^{c-\gamma-1}\exp\Big(-\frac{p}{t}-\frac{q}{(1-t)}\Big]\Big\}\notag\\
&\times E_{\alpha,\beta}^{c}(tz)dtdpdq.
\end{align}
Changing the order of integrations in equation (\ref{Mellin1}), we have
\begin{align}\label{Mellin2}
&\mathfrak{M}\Big\{E_{\alpha,\beta}^{\gamma,c}(z;p,q);p\rightarrow s,\rightarrow r\Big\}\notag\\
&=\frac{1}{\beta(\gamma,c-\gamma)}\int_0^1t^{\gamma-1}(1-t)^{c-\gamma-1}
E_{\alpha,\beta}^{c}(tz)\Big)\notag\\
&\times\Big[\int_0^\infty p^{s-1}\exp\Big(-\frac{p}{t}\Big)dp \int_{0}^{\infty}q^{r-1}\exp\Big(-\frac{q}{(1-t)}\Big)
dq\Big]dt.
\end{align}
Now, taking $u=\frac{p}{t}$  in  the second integral of equation (\ref{Mellin2}), we get
\begin{eqnarray*}
\int_0^\infty p^{s-1}\exp\Big(-\frac{p}{t}\Big)
dp &=&\int_0^\infty u^{s-1}t^s(1-t)^s e^{-u}du\\
&=& t^s(1-t)^s\int_0^\infty u^{s-1}e^{-u}du
\end{eqnarray*}
\begin{eqnarray}\label{Mellin3}
&=&t^s(1-t)^s\Gamma(s),
\end{eqnarray}
Similarly, taking $v=\frac{q}{1-t}$ in  the third integral of equation (\ref{Mellin2}), we have
\begin{eqnarray}\label{Mellin4}
\int_0^\infty q^{r-1}\exp\Big(-\frac{p}{1-t}\Big)
dq=(1-t)^r\Gamma(r).
\end{eqnarray}
Using equation (\ref{Mellin3}) and \eqref{a3} in equation (\ref{Mellin2}), we get
\begin{align*}
&\mathfrak{M}\Big\{E_{\alpha,\beta}^{\gamma,c}(z;p,q);p\rightarrow s,\rightarrow r\Big\}\\
&=\frac{\Gamma(s)\Gamma(r)}{\beta(\gamma,c-\gamma)}\int_0^1t^{\gamma+n+s-1}(1-t)^{c+r-\gamma-1}
\sum\limits_{n=0}^{\infty}\frac{(c)_nz^n}{\Gamma(\alpha n+\beta)n!}
dt
\end{align*}
which gives,
\begin{align*}
\mathfrak{M}\Big\{E_{\alpha,\beta}^{\gamma,c}(z;p,q);p\rightarrow s,\rightarrow r\Big\}&=\frac{\Gamma(s)\Gamma(r)}{\beta(\gamma,c-\gamma)}\sum\limits_{n=0}^{\infty}\frac{(c)_nz^n}{\Gamma(\alpha n+\beta)n!}\frac{\Gamma(\gamma+n+s)\Gamma(c+r-\gamma)}{\Gamma(c+n+s+r)}\\
&=\frac{\Gamma(s)\Gamma(r)\Gamma(c+r-\gamma)}{\Gamma(\gamma)\Gamma(c-\gamma)}
\sum\limits_{n=0}^{\infty}\frac{\Gamma(c+n)z^n}{\Gamma(\alpha n+\beta)n!}\frac{\Gamma(\gamma+n+s)}{\Gamma(c+n+s+r)}\\
&=\frac{\Gamma(s)\Gamma(r)\Gamma(c+r-\gamma)}{\Gamma(\gamma)\Gamma(c-\gamma)}{_2\Psi_2}\left[
         \begin{array}{cc}
           (c,1), (\gamma+s,1), & \\
            & ,z\\
           (\beta,\gamma),(c+s+r,1),&  \\
         \end{array}
       \right]
\end{align*}
which is the proof.
\end{proof}
\begin{corollary}\label{cor1}
Putting $s=r=1$ in Theorem \ref{thb}, we have the following result
\begin{eqnarray}
\int_{0}^{\infty}E_{\alpha,\beta}^{\gamma,c}(z;p,q)dpdq
=\frac{\Gamma(c+1-\gamma)}{\Gamma(\gamma)\Gamma(c-\gamma)}
._2\Psi_2\left[
         \begin{array}{cc}
           (c,1), (\gamma+1,1), & \\
            & ,z\\
           (\beta,\gamma),(c+2,1),&  \\
         \end{array}
       \right]
\end{eqnarray}
\end{corollary}
\begin{corollary}
Taking $p=q$ in Corollary \ref{cor1}, we get the following result of extended Mittag-Leffler(ML) function defined by \cite{OY}
\begin{eqnarray}
\int_{0}^{\infty}E_{\alpha,\beta}^{\gamma,c}(z;p)dp
=\frac{\Gamma(c+1-\gamma)}{\Gamma(\gamma)\Gamma(c-\gamma)}
{_2\Psi_2}\left[
         \begin{array}{cc}
           (c,1), (\gamma+1,1), & \\
            & ,z\\
           (\beta,\gamma),(c+2,1),&  \\
         \end{array}
       \right]
\end{eqnarray}
\end{corollary}
\begin{corollary}
 The inverse Mellin transform  of equation (\ref{Mell}) yields the following  integral representation

\begin{eqnarray*}
E_{\alpha,\beta}^{\gamma,c}(z;p,q)=\frac{1}{(2\pi\iota)^2\Gamma(\gamma)\Gamma(c-\gamma)}\int_{\mu-\iota\infty}^{\mu+\iota\infty}\int_{\nu-\iota\infty}^{\nu+\iota\infty}
\Gamma(s)\Gamma(r)\Gamma(c+r-\gamma)
\end{eqnarray*}
\begin{eqnarray}
\times{_2\Psi_2}\left[
         \begin{array}{cc}
           (c,1), (\gamma+s,1), & \\
            & ,z\\
           (\beta,\gamma),(c+s+r,1),&  \\
         \end{array}
       \right]p^{-s}q^{-r}dsdr,
\end{eqnarray}
where $\mu, \nu>0$, $\Re(p)>0$ and $\Re(q)>0$.
\end{corollary}
\section{Derivative properties and extended Mittag-Leffler(ML) function }
In this section, we present fractional derivatives by using the extended  Riemann-Liouville (R-L) fractional derivative and some related properties of \eqref{Fextended}.
\begin{definition}\label{defa}
The well-known R-L fractional derivative of order $\lambda$ is defined by
\begin{eqnarray}
\mathfrak{D}_{x}^{\lambda}\Big\{f(z)\Big\}=\frac{1}{\Gamma(-\lambda)}\int_0^xf(\tau)(x-\tau)^{-\lambda-1}d\tau, \Re(\lambda)>0.
\end{eqnarray}
If $m-1<\Re(\lambda)<m$ where $m=1,2,\cdots$, then it becomes
\begin{eqnarray*}
\mathfrak{D}_{x}^{\lambda}\Big\{f(z)\Big\}=\frac{d^m}{dx^m}\mathfrak{D}_{x}^{\lambda-m}\Big\{f(z)\Big\}
\end{eqnarray*}
\begin{eqnarray}
=\frac{d^m}{dx^m}\Big\{\frac{1}{\Gamma(-\lambda+m)}\int_0^xf(\tau)(x-\tau)^{-\lambda+m-1}d\tau\Big\}, \Re(\lambda)>0.
\end{eqnarray}
\end{definition}
\begin{definition}\label{defb}(see \cite{Ozerslan})
The extended R-L fractional derivative of order $\lambda$ is defined by
\begin{eqnarray}
\mathfrak{D}_{x}^{\lambda}\Big\{f(z);p\Big\}=\frac{1}{\Gamma(-\lambda)}\int_0^xf(\tau)(x-\tau)^{-\lambda-1}
\exp\Big(-\frac{px^2}{\tau(x-\tau)}\Big) d\tau, \Re(\lambda)>0.
\end{eqnarray}
if $m-1<\Re(\lambda)<m$ where $m=1,2,\cdots$, then it follows
\begin{eqnarray*}
\mathfrak{D}_{x}^{\lambda}\Big\{f(z);p,q\Big\}=\frac{d^m}{dx^m}\mathfrak{D}_{x}^{\lambda-m}\Big\{f(z);p,q\Big\}
\end{eqnarray*}
\begin{eqnarray}
=\frac{d^m}{dx^m}\Big\{\frac{1}{\Gamma(-\lambda+m)}\int_0^xf(\tau)(x-\tau)^{-\lambda+m-1}\exp\Big(-\frac{px^2}{\tau(x-\tau)}\Big)
 d\tau\Big\}, \Re(\lambda)>0.
\end{eqnarray}
\end{definition}
Recently, Baleanu et al. \cite{Baleanu2017} give the extension of R-L as
\begin{definition}\label{defc}
\begin{eqnarray}\label{Efrac}
\mathfrak{D}_{x}^{\lambda}\Big\{f(z);p,q\Big\}=\frac{1}{\Gamma(-\lambda)}\int_0^xf(\tau)(x-\tau)^{-\lambda-1}
\exp\Big(-\frac{px}{\tau}-\frac{qx}{(x-\tau)}\Big) d\tau,
\end{eqnarray}
where $\Re(\lambda)>0$, $\Re(p)>0$ and $\Re(q)>0$.
If  $m-1<\Re(\lambda)<m$ where $m=1,2,\cdots$, it follows
\begin{eqnarray*}
\mathfrak{D}_{x}^{\lambda}\Big\{f(z);p,q\Big\}=\frac{d^m}{dx^m}\mathfrak{D}_{x}^{\lambda-m}\Big\{f(z);p,q\Big\}
\end{eqnarray*}
\begin{eqnarray}
=\frac{d^m}{dx^m}\Big\{\frac{1}{\Gamma(-\lambda+m)}\int_0^xf(\tau)(x-\tau)^{-\lambda+m-1}
\exp\Big(-\frac{px}{\tau}-\frac{qx}{(x-\tau)}\Big)d\tau\Big\}, \Re(\lambda)>0.
\end{eqnarray}
\end{definition}
Obviously, if $p=q$, then definition \ref{defc} reduces to extended fractional derivative \ref{defb}. Similarly, if  $p=0=q$, then definition \ref{defc} reduces to R-L fractional derivative \ref{defa}.
\begin{theorem}
Let $p,q\geq0$, $\Re(\lambda)>\Re(\delta)>0$, $\Re(\alpha)>0$, and $\Re(\beta)>0$. Then
\begin{eqnarray}
\mathfrak{D}_{z}^{\delta-\lambda}\Big\{z^{\delta-1}E_{\alpha,\beta}^{c}(z);p,q\Big\}=\frac{z^{\lambda-1}\beta(\delta,c-\delta)}
{\Gamma(\mu-\delta)}E_{\alpha,\beta}^{\delta,\lambda}(z;p,q)
\end{eqnarray}
\end{theorem}
\begin{proof}
Setting $\lambda$ by $\delta-\lambda$ in definition  (\ref{Efrac}), we have
$$\mathfrak{D}_{z}^{\delta-\lambda}\Big\{z^{\delta-1}E_{\alpha,\beta}^{c}(z);p,q\Big\}$$
\begin{eqnarray*}
&=&\frac{1}{\Gamma(\lambda-\delta)}\int_0^z\tau^{\delta-1}E_{\alpha,\beta}^{c}(\tau)(z-\tau)^{-\delta+\mu-1}
\exp\Big(-\frac{pz}{\tau}-\frac{qz}{(z-\tau)}\Big)d\tau\\
&=&\frac{z^{-\delta+\mu-1}}{\Gamma(\mu-\delta)}\int_0^z\tau^{\delta-1}E_{\alpha,\beta}^{c}(\tau)(1-\frac{\tau}{z})^{-\delta+\lambda-1}
\exp\Big(-\frac{pz}{\tau}-\frac{qz}{(z-\tau)}\Big) d\tau.
\end{eqnarray*}
Taking $u=\frac{\tau}{z}$ in above equation, we have
$$\mathfrak{D}_{z}^{\delta-\mu}\Big\{z^{\delta-1}E_{\alpha,\beta}^{c}(z);p,q\Big\}$$
\begin{multline}\label{Frac}
=\frac{z^{\mu-1}}{\Gamma(\mu-\delta)}\int_0^z u^{\delta-1}(1-u)^{-\delta+\lambda-1}\exp\Big(-\frac{p}{u}-\frac{q}{(1-u)}\Big) E_{\alpha,\beta}^{c}(uz)du.
\end{multline}
Comparing equation (\ref{Frac}) with equation (\ref{integral}), then completes the proof.
\end{proof}
Now, we define the following  derivative properties of extended $(p,q)$-Mittag-Leffler function.
\begin{theorem}
  The following derivative formula holds for the extended Mittag-Leffler(ML) function,
\begin{eqnarray}\label{der}
\frac{d^n}{dz^n}\Big\{E_{\alpha,\beta}^{\gamma,c}(z;p,q)\Big\}
=(c)_nE_{\alpha,\beta+n\alpha}^{\gamma+n,c+ n}(z;p,q).
\end{eqnarray}
\end{theorem}
\begin{proof}
Taking derivative of equation (\ref{Fextended}) respect to $z$, then we have
\begin{eqnarray}\label{der1}
\frac{d}{dz}\Big\{E_{\alpha,\beta}^{\gamma,c}(z;p,q)\Big\}
=cE_{\alpha,\beta+\alpha}^{\gamma+1,c+ 1}(z;p,q).
\end{eqnarray}
Again taking derivative of equation (\ref{der1}), respect to $z$, we have
\begin{eqnarray}
\frac{d^2}{dz^2}\Big\{E_{\alpha,\beta}^{\gamma,c}(z;p,q)\Big\}
=c(c+1)E_{\alpha,\beta+2\alpha}^{\gamma+2,c+ 2}(z;p,q).
\end{eqnarray}
Continuing in this way up to $n$, we obtain the result.
\end{proof}
\begin{theorem}
 Following differentiation formula hold true:
\begin{eqnarray}
\frac{d^n}{dz^n}\Big\{z^{\beta-1}E_{\alpha,\beta}^{\gamma,c}(\mu z^\alpha;p,q)\Big\}
=z^{\beta-n-1}E_{\alpha,\beta-n}^{\gamma+n,c+ n}(\mu z^\alpha;p,q).
\end{eqnarray}
\end{theorem}
\begin{proof}
Replacing $z$ by $\mu z^\alpha$ in equation (\ref{der}) and multiply by $z^{\beta-1}$ and differentiate $n$ times respect to $z$, we obtain the required result.
\end{proof}

\section{conclusion}
In this study, we established a new extension of Mittag-Leffler(ML) and some of its results. We conclude that if $p=q$, then we get the results of extended Mittag-Leffler(ML) function defined in \cite{OY}. Furthermore, if $p=q=0$, then we have the classical Mittag-Leffler(ML) function.\\

\noindent{\bf Conflict of Interests:} \vskip 2mm

\noindent There is no conflict of interests.\vskip 2mm

\noindent{\bf Authors’ contributions:}\vskip 2mm

\noindent The authors contributed equally. All the authors read and approved the final manuscript.\\



\begin{thebibliography}{99}
\bibitem[1]{Baleanu2017}D. Baleanu, P. Agarwal, R. K. Parmar, M. M. Alqurashi, S. Salahshour, \emph{Extension of the fractional derivative operator of the Riemann-Liouville}, J. Nonlinear Sci. Appl., 10 (2017), 2914-2924
\bibitem[2]{CQ} M. A. Chaudhry, A. Qadir, H. M. Srivastava, R. B. Paris, \emph{ Extended hypergeometric and confluent hypergeometric functions}.
Appl. Math. Comput. 159, 589-602, (2004).
\bibitem[3]{Choi2014} J. Choi, A. K. Rathie, R. K. Parmar, \emph{Extension of extended beta, hypergeometric and confluent hypergeometric functions}, Honam Mathematical J. 36 (2014), No. 2, pp. 357-385.
\bibitem[4]{R6} R. Gorenflo, A.A. Kilbas, S.V. Rogosin, \emph{On the generalized Mittag-Leffler type functions}, Integral Transform. Spec. Funct. 7, 215-224, (1998).
%
\bibitem[5]{Kurulay} M. Kurulay, M. Bayram, \emph{Some properties of Mittag-Leffler function and their relation with the Wright function}, Adv. Differ. Equ. 2012, 178(2012).

\bibitem[6]{FM} F. Mainardi,  \emph{On some properties of the Mittag-Leffler function, $E_\rho(tz)$, completely monotone for $t > 0$ with $0 < \rho < 1$}, arXiv:1305.0161

\bibitem[7]{OY}M. A. \"{O}zarslan, B. Y{\i}lmaz, \emph{The extended Mittag-Leffler function and its properties}, Journal of Inequalities
and Applications, 2014:85, (2014).
\bibitem[8]{Ozerslan} M. A. \"{O}zarslan, E. \"{O}zergin, \emph{Some generating relations for extended hypergeometric functions via generalized fractional derivative operator}, Math. Comput. Model. 52(2010), 1825-1833.
\bibitem[9]{R16} T.R. Prabhakar, \emph{A singular integral equation with a generalized Mittag-Leffler function in the kernel}, Yokohama Math. J. 19, 7-15, (1971).
\bibitem[10]{Rahman1}G. Rahman, P. Agarwal, S. Mubeen, M. Arshad, \emph{Fractional integral operators involving extended
Mittag–Leffler function as its Kernel}, Bol. Soc. Mat. Mex. (2017), 1-14.
\bibitem[11]{Rahman2}G. Rahman, D. Baleanu, M. Al-Qurashi, S. .D Purohit, S. Mubeen, M. Arshad, \emph{The extended Mittag-Leffler function via fractional calculus}, J. Nonlinear Sci. Appl., 10 (2017), 4244-4253.
\bibitem[12]{Rahman3}G. Rahman, A. Ghaffar, S. Mubeen, M. Arshad, S.U. Khan, \emph{The composition of extended
Mittag-Leffler functions with pathway integral operator}, Advances in Difference Equations  (2017) 2017:176.
\bibitem{rahman-praveen} G. Rahman, P. Agarwal,  S. et al., \emph{Fractional integral operators involving extended Mittag-Leffler function as its Kernel}, Bol. Soc. Mat. Mex. (2017). https://doi.org/10.1007/s40590-017-0167-5
\bibitem[13]{CP}A. K. Shukla and J. C. Prajapati, \emph{On a generalization of Mittag-Leffler functions and its properties}, J. Math. Anal. Appl., 337,797-811 (2007).

\bibitem[14]{R21} H.M. Srivastava, \emph{A contour integral involving Fox’s H-function}, Indian J. Math. 14, 1-6, (1972).
\bibitem[15]{R22} H.M. Srivastava, \emph{A note on the integral representation for the product of two generalized Rice polynomials}, Collect. Math. 24 117-121, (1973).
\bibitem[16]{R23} H.M. Srivastava, K.C. Gupta, S.P. Goyal, \emph{The H-Functions of One and Two Variables with Applications}, South Asian Publishers, New Delhi and Madras, (1982).
\bibitem[17]{R24} H.M. Srivastava, C.M. Joshi, \emph{Integral representation for the product of a class of generalized hypergeometric polynomials}, Acad. Roy. Belg. Bull. Cl. Sci. (Ser. 5) 60,  919-926, (1974).
\bibitem[18]{SZ}H.M. Srivastava , \v{Z}. Tomovski, \emph{Fractional calculus with an integral operator containing a generalized
Mittag–Leffler function in the kernel}, Applied Mathematics and Computation 211,  189-210, (2009).
\bibitem[19]{Wright1935} E.M. Wright, \emph{The asymptotic expansion of the generalized hypergeometric functions}, J. London. Math. Soc. 10(1935), pp. 286-293.
\bibitem[20]{Wrig1935} E.M. Wright, \emph{The asymptotic expansion of integral functions defined by Taylor Series}, Philos. Trans. Roy. Soc. London A 238 (1940), pp. 423-451.
\bibitem[21] {Wrigt1935} E.M. Wright, \emph{The asymptotic expansion of the generalized hypergeometric function II}, Proc. London. Math. Soc. 46(2) (1935), pp. 389-408.
\end{thebibliography}
\end{document}